\documentclass[11pt,oneside]{article}
\usepackage[latin1]{inputenc}
\usepackage[T1]{fontenc}
\usepackage{ae}
\usepackage{aecompl}
\usepackage{amsmath}
\usepackage{amsthm}
\usepackage{amsfonts}
\usepackage{amssymb}
\usepackage{mathrsfs}
\usepackage[dvips]{graphicx}
\usepackage{float}
\usepackage[all]{xy}
\usepackage{color}
\theoremstyle{plain}
\newtheorem{prop}{Proposition}[section]
\newtheorem{teo}[prop]{Theorem}

\newtheorem{fact}[prop]{Fact}
\newtheorem{lem}[prop]{Lemma}

\theoremstyle{remark}

\definecolor{red}{rgb}{1,0,0}
\definecolor{green}{rgb}{0,1,0.2}

\title{Components of moduli spaces of spin curves with the expected codimension}
\author{Luca Benzo}
\date{}
\begin{document}
\maketitle
\begin{abstract}
\noindent We prove a conjecture of Gavril Farkas claiming that for all integers $r \geq 2$ and $g \geq \binom{r+2}{2}$ there exists a component of the locus $\mathcal{S}^r_g$ of spin curves with a theta characteristic $L$ such that $h^0(L) \geq r+1$ and $h^0(L)\equiv r+1 (\text{mod } 2)$ which has codimension $\binom{r+1}{2}$ inside the moduli space $\mathcal{S}_g$ of spin curves of genus $g$.
\end{abstract}
\begin{section}{Introduction}
Let $g \geq 1$ be an integer and let $\mathcal{S}_g$ be the moduli space of smooth spin curves parameterizing pairs $[C,L]$ where $C$ is a smooth curve of genus $g$ and $L$ is a theta-characteristic i.e. a line bundle on $C$ such that $L^2 \cong \omega_C$.
Let $r \geq 0$ be an integer and consider the locus
$$\mathcal{S}^{r}_{g} \cong \left\{[C,L] \in \mathcal{S}_g : h^{0}(L) \geq r+1 \hbox{ and } h^{0}(L) \equiv r+1 (\text{mod } 2)\right\}.$$
In \cite{HA} Harris proved that each component of $\mathcal{S}^r_g$ has dimension $\geq 3g-3- \binom{r+1}{2}$. Since the forgetful map $\mathcal{S}_g \rightarrow \mathcal{M}_g$ is finite, this is equivalent to say that each component of $\mathcal{S}^r_g$ has codimension less than or equal to $\binom{r+1}{2}$ in $\mathcal{S}_g$. If the codimension in $\mathcal{S}_g$ of a component $V_r \subset \mathcal{S}^r_g$ equals $\binom{r+1}{2}$, we also say that $V_r$ has the \emph{expected codimension} (in $\mathcal{S}_g$). Of course, the problem of finding components of the expected codimension makes sense only if $3g-3- \binom{r+1}{2} \geq 0$ i.e. when $g$ is not too small with respect to $r$.\\
Many results are known for small $r$. Indeed, it is known that $\mathcal{S}^r_g$ has pure codimension $\binom{r+1}{2}$ in $\mathcal{S}_g$ for all $g$ if $r=1,2$, and for all $g \geq 8$ if $r=3$ (cf. \cite{BEA} and \cite{TEI}).
Moreover, in \cite{FA} Farkas showed, for $3 \leq r \leq 11$, $r \neq 10$ and $g \geq g(r)$ ($g(r)$ is a positive integer depending on $r$), the existence of a component $V_r \subset \mathcal{S}^r_g$ having codimension $\binom{r+1}{2}$ in $\mathcal{S}_g$ and such that, for a general point $[C,L]$ in $V_r$, $L$ is very ample. His proof consists in two parts. The first one is an inductive argument stating that, for fixed integers $r, g_0$, the existence of a component of $\mathcal{S}^r_{g_0}$ of the expected codimension implies, for all $g \geq g_0$, the existence of a component of $\mathcal{S}^r_g$ of the expected codimension. The second is the explicit construction, for the above mentioned values of $r$ and a suitable integer $g(r)$, of a component of $\mathcal{S}^r_{g(r)}$ with the expected codimension.\\
In the same paper, the author conjectures the existence for all $r$ of a component of $\mathcal{S}^r_{g}$, $g \doteq \binom{r+2}{2}$ with the expected codimension. We will discuss in the next section how this value of $g$ comes out.\\
The purpose of the present paper is to prove this conjecture. We first prove the following theorem by induction on $r$
\begin{teo}
For all integers $r \geq 2$ and $g(r) \doteq {r+2 \choose 2}$, the locus $\mathcal{S}^r_{g(r)}$ has a component of codimension ${r+1 \choose 2}$ in $\mathcal{S}_{g(r)}$.
\end{teo}
This result, combined with Farkas' inductive argument, gives the following
\begin{teo}
For all integers $r \geq 2$ and $g \geq g(r) \doteq {r+2 \choose 2}$, the locus $\mathcal{S}^r_{g}$ has a component of codimension ${r+1 \choose 2}$ in $\mathcal{S}_{g}$.
\end{teo}
If $C \subset \mathbb{P}^r$ is a nondegenerate curve, we will denote by $N_C$ (respectively, $\mathcal{I}_C$) the normal sheaf (respectively, ideal sheaf) of $C$ in $\mathbb{P}^r$. We work over the field of complex numbers.\\
This research project was partially supported by FIRB 2012 "Moduli spaces and Applications".
\end{section}
\begin{section}{Preliminary results}
Let $C$ be a smooth curve and $L$ a line bundle on it. Let $\Phi_L$ be the Gaussian map of $L$ and let $\Psi_L \doteq {\Phi_L}_{|\wedge^2 H^{0}(L)}:\wedge^2 H^{0}(L) \rightarrow H^{0}(\omega_C \otimes L^2)$ (see \cite{WA} or \cite{FA}, Section 3 for details). If $L$ is very ample, identifying $C$ with its embedding in $\mathbb{P}^r$ by the map $\varphi_L$ defined by $L$ one has
\begin{equation}
H^{0}(N^{\vee}_C(2)) = \ker \Phi_L=\ker \Psi_L \oplus S_2(L)
\end{equation}
where $S_2(L) \doteq \ker \left\{ \text{Sym}^2 H^{0}(L) \rightarrow H^{0}(L^2)\right\} = H^{0}(\mathcal{I}_C(2))$, hence $\Psi_L$ is injective if and only if $H^{0}(\mathcal{I}_C(2)) \cong H^{0}(N^{\vee}_C(2))$.\\
The map $\Psi_L$ has a very interesting relation with spin curves. For a point $[C,L] \in \mathcal{S}^r_g$, the forgetful map $\pi:\mathcal{S}_g \rightarrow \mathcal{M}_g$ gives a natural identification $T_{[C,L]}\mathcal{S}_g = T_{[C]}\mathcal{M}_g=H^{0}(\omega^2_C)^{\vee}$. Nagaraj (\cite{NAG}) has shown that
$$T_{[C,L]}\mathcal{S}^r_g = \left( \text{Im}(\Psi_L)\right)^{\perp}$$
where $\perp$ denotes the orthogonal complement in $T_{[C,L]}\mathcal{S}_g$.\\
If the map $\Psi_L$ is injective, then $\dim \left( \text{Im}(\Psi_L)\right)^{\perp} =3g-3- \dim (\wedge^2 H^{0}(L))=3g-3-\binom{r+1}{2}$,
(note that in this case $h^{0}(L)=r+1$, otherwise Harris' bound would be violated)
hence to show that a component $V_r \subset \mathcal{S}^{r}_{g}$ has the expected codimension, it is sufficient to find a pair $[C,L] \in V_r$ such that
$\Psi_L$ is injective. As a consequence, one immediately obtains the following
\begin{fact}
\label{ineq}
Let $C \subset \mathbb{P}^r$ be a smooth half-canonical curve of genus $g$. Then one has that $h^{0}(\mathcal{I}_C(2)) \leq h^{0}(N^{\vee}_C(2))=h^1(N_C)$. Moreover, if equality holds (in particular if $h^1(N_C)=0$), then $[C,\mathcal{O}_C(1)]$ belongs to a component of $\mathcal{S}^{h^{0}(\mathcal{O}_C(1))-1}_g$ having the expected codimension in $\mathcal{S}_g$.
\end{fact}
Suppose that $C \subset \mathbb{P}^r$ is a smooth half-canonical curve of genus $g$ such that $h^{0}(\mathcal{O}_C(1))=r+1$.
If $g < \binom{r+2}{2}$, then the exact sequence
$$0 \rightarrow \mathcal{I}_C(2) \rightarrow \mathcal{O}_{\mathbb{P}^{r}}(2) \rightarrow \mathcal{O}_C(2) \rightarrow 0$$
gives $h^{0}(\mathcal{I}_C(2)) >0$, hence by Fact \ref{ineq} one has $h^1(N_C) > 0$. This makes extremely difficult to solve the problem by degeneration arguments i.e. exhibiting $C$ as a deformation of a reducible curve $X \subset \mathbb{P}^r$ because these deformations are possibly obstructed.\\
On the other hand, if $g > \binom{r+2}{2}$ and $H^1(N_C)=(0)$, then the number of moduli of the component $W \subset Hilb^r_{g,g-1}$ parameterizing the curve $C$ is greater than or equal to $\min \left\{3g-3,3g-3+\rho(g,r,g-1)\right\} > 3g-3-\binom{r+1}{2}$ (see \cite{Sernesi} for more details). Nevertheless, by Fact \ref{ineq}, the point $[C,\mathcal{O}_C(1)]$ belongs to a component of $\mathcal{S}^r_g$ which has the expected codimension. This implies that $W$ does not consist entirely of half-canonically embedded curves, hence $C$ cannot be obtained as a general deformation of a curve $X$ as above.\\
The value $g=\binom{r+2}{2}$ is the \emph{only} value for which a regular component of $Hilb^r_{g,g-1}$ which consists entirely of half-canonically embedded curves can exist. In the next section we will prove, by induction on $r$, that for every $r \geq 2$ there exists a component $V_r \subset \mathcal{S}^r_{\binom{r+2}{2}}$ with the expected codimension and such that, for a general point $[C,L] \in V_r$, the line bundle $L$ is very ample. The result which will be used as the base case for the induction is the following
\begin{prop}
\label{r=2}
The locus $\mathcal{S}^2_{6}$ has a component $V$ of codimension $3$ in $\mathcal{S}_6$ such that, for a general point $[C,L] \in V$, the line bundle $L$ is very ample and $h^{1}(N_C)=0$ in the embedding of $C$ defined by $L$.
\end{prop}
\begin{proof}
Let $C$ be a smooth plane quintic and $L=\mathcal{O}_C(1)$. The adjunction formula gives that $L$ is a theta-characteristic and $h^{1}(N_C)=0$.
\end{proof}
In order to prove the conjecture we need also the following fundamental result due to Farkas, which is proved by inductively smoothing a stable spin structure (see \cite{COR} for details) over a suitable reducible curve.
\begin{prop}[\cite{FA}, Proposition 2.4]
\label{farkasinductive}
Fix integers $r,g_0 \geq 1$. If $\mathcal{S}^r_{g_0}$ has a component of codimension $\binom{r+1}{2}$ in $\mathcal{S}_{g_0}$, then for every $g \geq g_0$ the locus $\mathcal{S}^r_g$ has a component of codimension $\binom{r+1}{2}$ in $\mathcal{S}_g$.
\end{prop}
The core of the proof of our result is the construction of a reducible curve $X \subset \mathbb{P}^r$ which can be smoothed to a curve with a theta characteristic and having nice properties ensuring expected codimension of the component of $\mathcal{S}^r_g$ parameterizing it. To this purpose, we briefly recall some basic facts about deformations of embedded curves and then state two general lemmas.\\ \\
Let $X \subset \mathbb{P}^r$, $r \geq 2$, be a connected reduced curve with only nodes as singularities. There is a 4-terms exact sequence
$$0 \rightarrow T_X \rightarrow {T_{\mathbb{P}^r}}_{|X} \rightarrow N_X \rightarrow T^1_X \rightarrow 0$$
where $T_X \doteq \textbf{Hom}(\Omega^1_X,\mathcal{O}_X)$ and $T^1_X$ is the \emph{first cotangent sheaf} of $X$, which is a torsion sheaf supported on the singular locus $S$ of $X$ and having stalk $\mathbb{C}$ at each of the points. The sheaf $N'_X \doteq \ker \left\{N_X \rightarrow T^1_X \right\}$ is called the \emph{equisingular normal sheaf} of $X$.\\
For each $p \in S$ let $T^1_p$ denote the restriction to $p$ of $T^1_X$ extended by zero on $X$.\\
Moreover, from the deformation-theoretic interpretation of the cohomology spaces associated to these sheaves, it follows that, if the cohomology maps $H^{0}(N_X) \rightarrow H^{0}(T^1_p)$ are surjective for every $p \in S$ and $H^{1}(N_X)=(0)$, then the curve $X$ (flatly) deforms to a smooth curve in $\mathbb{P}^r$ i.e. it is \emph{smoothable in $\mathbb{P}^r$}. Note that in particular $X$ is smoothable if $H^{1}(N'_X)=(0)$ (see \cite{Sernesi}, Section 1 for an extensive discussion of all this topic).
\begin{lem}[\cite{Sernesi}, Lemma 5.1]
Let $r \geq 2$ be an integer, let $C$ and $D$ be smooth curves in $\mathbb{P}^r$ intersecting transversally in a smooth $0$-dimensional scheme $\Delta$ and let $X \doteq C \cup D$. Then there are exact sequences
\begin{equation}
\label{seq1}
0 \rightarrow {N_X}_{|D}(-\Delta) \rightarrow N'_X \rightarrow N_C \rightarrow 0
\end{equation}
\begin{equation}
\label{seq2}
0 \rightarrow N_D(-\Delta) \rightarrow {N_X}_{|D}(-\Delta) \rightarrow T^1_X \rightarrow 0.
\end{equation}
\end{lem}
\begin{lem}
\label{critical}
Let $r \geq 2$ be an integer, let $X=\mathcal{X}_0 \subset \mathbb{P}^r$ be a nodal curve and let $q: \mathcal{X} \subset \mathbb{P}^r \times B \rightarrow B$ be a (flat) family of locally trivial deformations of $X$ over an irreducible algebraic scheme $B$. Assume that the \emph{critical scheme} $\mathcal{S} \subset \mathcal{X}$ of the map $q$, i.e. the subscheme of nodal points of the fibres of $q$, is irreducible. Moreover, assume that $H^{1}(N_X)=(0)$ and $h^{0}(N'_X) < h^{0}(N_X)$. Then the curve $X$ is smoothable in $\mathbb{P}^r$.
\end{lem}
\begin{proof}
Up to shrinking $B$, we can assume that $B$ is nonsingular and $\mathcal{S} \rightarrow B$ is an \'etale covering of degree $|\text{Sing}(X)|$. Then consider the base change
$$\xymatrix{\widetilde{\mathcal{X}} \ar[d]^{\widetilde{q}} \ar[r] & \mathcal{X} \ar[d]^{q} \\ \mathcal{S} \ar[r] & B}$$
By construction, there exists a section $\sigma:\mathcal{S} \rightarrow \widetilde{\mathcal{X}}$ of $\widetilde{q}$.\\
Let $\mathcal{N}_{\widetilde{X}/S}$ (respectively, $\mathcal{T}^1_{\widetilde{X}/S}$) be the relative normal bundle (respectively, the first relative cotangent sheaf) of $\widetilde{X}$ over $S$ (see \cite{RA}, Section 2 for a precise definition), and consider the map $\widetilde{q}_{*}\mathcal{N}_{\widetilde{X}/S} \xrightarrow{\rho} \sigma^{*}\mathcal{T}^1_{\widetilde{X}/S} \rightarrow \text{coker} \rho \rightarrow 0$.\\
Since $h^{0}(N'_X) < h^{0}(N_X)$, the image of the map $H^{0}(N_X) \rightarrow H^{0}(T^1_X)$ is at least one-dimensional, thus there is at least one $p \in \text{Sing}(X)$ such that the map $H^{0}(N_X) \rightarrow H^{0}(T^1_p)$ is surjective. If $s=(q(\mathcal{X}_0),p) \in \mathcal{S}$, that map is exactly the fibre map
$${\widetilde{q}_{*}\mathcal{N}_{\widetilde{X}/S}}_{|s} \cong H^{0}(N_X) \xrightarrow{\rho_s} {\sigma^{*}\mathcal{T}^1_{\widetilde{X}/S}}_{|s} \cong H^{0}(T^1_p)$$
where the isomorphisms are given by the fact that the cohomology of both sheaves is constant on fibres and both sheaves are flat over $\mathcal{S}$. Then, since $\text{coker} \rho$ is a coherent sheaf, there exists a nonempty open set $\mathcal{U} \subset \mathcal{S}$ such that ${\text{coker} \rho}_{|\mathcal{U}} \cong (0)$.
Since $\mathcal{S}$ is irreducible, this implies that, for the general fibre $\widetilde{\mathcal{X}}_t$, the map $H^{0}(N_{\widetilde{\mathcal{X}}_t}) \rightarrow H^{0}(T^1_w)$ is surjective for every $w \in \text{Sing}(\widetilde{\mathcal{X}}_t)$. Since $H^{1}(N_{\widetilde{\mathcal{X}}_t})=(0)$ by upper semicontinuity of the cohomology, it follows that $\widetilde{\mathcal{X}}_t$ is smoothable, hence $X$ is too.
\end{proof}
\end{section}
\begin{section}{The main result}
We first need two other lemmas.
\begin{lem}
\label{h1nd-delta}
Let $r \geq 2$ be an integer, $D \subset \mathbb{P}^r$ a sufficiently general elliptic normal curve (of degree $r+1$) and $\Delta$ a hyperplane section of $D$. Then $H^{1}(N_D(-\Delta))=(0)$.
\end{lem}
\begin{proof}
Let $R \subset \mathbb{P}^r$ be a rational normal curve, let $l \subset \mathbb{P}^r$ be a line intersecting $R$ in a smooth 0-dimensional scheme $S$ of length 2 and let $Y \doteq R \cup l$. Let $H$ be a hyperplane section of $Y$. Since $N_R \cong \mathcal{O}_{\mathbb{P}^1}(r+2)^{\oplus r-1}$, tensorizing the fundamental exact sequence
$$0 \rightarrow \mathcal{O}_R(-S) \rightarrow \mathcal{O}_Y \rightarrow \mathcal{O}_l \rightarrow 0$$
by ${N_Y}(-H)$ and using suitable twistings of sequence (\ref{seq2}) on $l$ and $R$, it follows that $H^1(N_Y(-H))=(0)$. Moreover, using again sequence (\ref{seq2}) on $l$ and sequence (\ref{seq1}), one proves that $H^1(N'_Y)=(0)$, hence $Y$ is smoothable to an elliptic normal curve $D$ having $H^{1}(N_D(-\Delta))=(0)$.
\end{proof}
\begin{lem}
\label{h1nxc}
Let $r \geq 2$ be an integer, let $H \subset \mathbb{P}^{r+1}$ be a hyperplane and let $C \subset H$ be a smooth curve such that $h^1(\mathcal{O}_C(1))\leq r+1$ and $H^1(N_{C/H})=(0)$. Define $N_C \doteq N_{C/H}$ and $\widetilde{N}_C \doteq N_{C/\mathbb{P}^{r+1}}$. Let $D \subset \mathbb{P}^{r+1}$ be a sufficiently general elliptic normal curve (of degree $r+2$) intersecting $C$ transversally in a general smooth $0$-dimensional subscheme $\Delta$ of length $r+2$. Let $X \doteq C \cup D$. Then the curve $X$ is smoothable in $\mathbb{P}^{r+1}$.
\end{lem}
\begin{proof}
Let $\Sigma$ be a hyperplane section of $C$ and consider the commutative exact diagram
\begin{equation}
\xymatrix{ & 0 \ar[d] & 0 \ar[d] & &\\ 0 \ar[r] & N_{C} \ar[r]^{\cong} \ar[d]^{\alpha} & N_{C} \ar[r] \ar[d]^{\gamma} & 0  \ar[d] & \\ 0 \ar[r]  & \widetilde{N}_C \ar[r]^{\beta} \ar[d] & {N_{X}}_{|C} \ar[r] \ar[d] & T^1_X \ar[r] \ar[d]^{\cong} & 0 \\0 \ar[r] & \mathcal{O}_C(\Sigma) \ar[d] \ar[r] & \mathcal{A} \ar[r] \ar[d] & \mathcal{O}_{\Delta} \ar[d] \ar[r] & 0 \\ & 0 & 0 & 0 &}
\end{equation}
(the central horizontal sequence is given by tensorizing sequence (\ref{seq2}) on $C$ by $\mathcal{O}_X(\Delta)$).
Since $\deg {N_X}_{|C} - \deg N_C= \deg C+ r+2$, if the sheaf $\mathcal{A}$ is a line bundle, then it must be isomorphic to $\mathcal{O}_C(\Sigma+\Delta)$. Note that $\mathcal{A}$ is a line bundle if and only if the fibre map $\gamma_x \doteq \beta_x \circ \alpha_x: {N_C}_{|x} \rightarrow {N_X}_{|x}$ is injective for all $x \in C$. Note also that $\alpha_x$ is injective for all $x \in C$. Consider the exact sequence of vector spaces
$${\widetilde{N_C}}_{|x} \xrightarrow{\beta_x} {N_X}_{|x} \rightarrow {T^1_X}_{|x} \rightarrow 0.$$
If $x \notin \Delta$, then ${T^1_X}_{|x}=(0)$, hence $\beta_x$ and $\gamma_x$ are injective. If $x \in \Delta$, then $\ker \beta_x = T_{x}D$, where $T_{x}D$ is the tangent line to the curve $D$ at $x$. Since $D$ intersects $H$ transversally, $T_{x}D$ does not lie in the tangent space to $H$ at $x$, hence $\ker \beta_x \cap \text{im }\alpha_x=(0)$ and the composition map $\gamma_x$ is again injective.\\
Since $\Delta$ is general and $h^{1}(\mathcal{O}_C(\Sigma))\leq r+1$, by the geometric version of Riemann-Roch theorem one has that $H^{1}(\mathcal{O}_C(\Sigma+\Delta))=(0)$, hence the vertical central sequence in the diagram gives $H^{1}({N_X}_{|C})=(0)$.\\
Consider the exact sequence
$$0 \rightarrow {N_X}_{|D}(-\Delta) \rightarrow N_X \rightarrow {N_X}_{|C} \rightarrow 0.$$
Since $H^{1}(N_D(-\Delta))=(0)$ by Lemma \ref{h1nd-delta}, sequence (\ref{seq2}) gives $H^1({N_X}_{|D}(-\Delta))=(0)$, hence $H^1(N_X)=(0)$.\\
Moreover, using sequence (\ref{seq1}), one obtains $h^1(N'_X)=h^{1}(\widetilde{N}_C) \leq r+1$, the last inequality being given by the left vertical sequence of the diagram. It then follows from the exact sequence
$$0 \rightarrow N'_X \rightarrow N_X \rightarrow T^1_X \rightarrow 0$$
that $h^{0}(N'_X) < h^{0}(N_X)$. Let $B$ be a component of the subscheme of the Hilbert scheme parameterizing locally trivial deformations of $X$, and let $u:\mathcal{X} \rightarrow B$ be the universal family. Since the critical scheme $\mathcal{S} \subset \mathcal{X}$ of $u$ is irreducible,  Lemma \ref{critical} applies and gives the smoothability of $X$ in $\mathbb{P}^{r+1}$.
\end{proof}
\begin{teo}
\label{ind}
For all integers $r \geq 2$ and $g(r) \doteq {r+2 \choose 2}$, the locus $\mathcal{S}^r_{g(r)}$ has a component $V_r$ of codimension ${r+1 \choose 2}$ in $\mathcal{S}_{g(r)}$ such that, for a general point $[C,L] \in V_r$, the line bundle $L$ is very ample and $h^{1}(N_C)=0$ in the embedding defined by $L$.
\end{teo}
\begin{proof}
The proof is by induction on $r$. The base case is given by Proposition \ref{r=2}. Let $[C,L] \in V_r$ be a general point. Then by the inductive assumption $C = \varphi_{L}(C) \subset  H \cong \mathbb{P}^r$ is a smooth curve of genus $g = g(r)$ and degree $g-1$ such that $h^{1}(N_C)=0$.\\
Embed $H$ as a hyperplane in $\mathbb{P}^{r+1}$ and let $\Delta$ be a general smooth $0$-dimensional subscheme of $C$ of length $r+2$. Let $D \subset \mathbb{P}^{r+1}$ be a sufficiently general elliptic normal curve (of degree $r+2$) intersecting $C$ transversally in $\Delta$ and let $X \doteq C \cup D$.
One has $p_a(X)=g+r+2=\binom{r+3}{2}$ and $\deg \mathcal{O}_X(1) = p_a(X)-1$.\\
Let $\widetilde{\mathcal{I}}_C \doteq \mathcal{I}_{C/\mathbb{P}^{r+1}}$, $\widetilde{N}_C \doteq N_{C/\mathbb{P}^{r+1}}$ and $\mathcal{I}_{\Delta} \cong \mathcal{I}_{\Delta/\mathbb{P}^{r+1}}$. We first want to show that $h^{1}(\mathcal{O}_X(2))=1$. The cohomology sequence of
\begin{equation}
\label{ix2opr+12}
0 \rightarrow \mathcal{I}_X(2) \rightarrow \mathcal{O}_{\mathbb{P}^{r+1}}(2) \rightarrow \mathcal{O}_X(2) \rightarrow 0
\end{equation}
gives $h^{1}(\mathcal{O}_X(2))=h^{2}(\mathcal{I}_X(2))$. Consider the Mayer-Vietoris sequence
\begin{equation}
\label{mvx}
0 \rightarrow \mathcal{I}_X(2) \rightarrow \widetilde{\mathcal{I}}_C(2) \oplus \mathcal{I}_D(2) \rightarrow \mathcal{I}_{\Delta}(2) \rightarrow 0.
\end{equation}
Since $r+2$ points always impose independent conditions to quadrics in $\mathbb{P}^{r+1}$, the cohomology sequence associated to
$$0 \rightarrow \mathcal{I}_{\Delta}(2) \rightarrow \mathcal{O}_{\mathbb{P}^{r+1}}(2) \rightarrow \mathcal{O}_{\Delta}(2) \rightarrow 0$$
gives $h^{1}(\mathcal{I}_{\Delta}(2))=0$.\\
Sequences analogous to (\ref{ix2opr+12}) for $D$ and $C$ give $h^{2}(\mathcal{I}_D(2))=0$ and $h^2(\widetilde{\mathcal{I}}_C(2))=h^{1}(\mathcal{O}_C(2))=1$, where the last equality holds by inductive assumption, hence (\ref{mvx}) gives $h^{1}(\mathcal{O}_X(2))=1$.\\
Note that every quadric hypersurface in $\mathbb{P}^{r+1}$ containing $C$ is reducible, thus $h^{0}(\mathcal{I}_X(2))=0$.
Moreover, by Riemann-Roch one has $h^{0}(\mathcal{O}_X(2))=p_a(X)$, hence (\ref{ix2opr+12}) gives $h^1(\mathcal{I}_X(2))=0$.\\
By Lemma \ref{h1nxc} one has that $X$ deforms to a smooth curve $\Gamma$ of geometric genus $p_a(X)=\binom{r+3}{2}$ and such that $h^1(N_{\Gamma})=0$. Since $\chi(\mathcal{I}_{\Gamma}(2))=\chi(\mathcal{I}_X(2))$ and $h^1(\mathcal{I}_X(2))=0$, the upper semicontinuity of the cohomology gives $h^2(\mathcal{I}_{\Gamma}(2))=h^{1}(\mathcal{O}_{\Gamma}(2))=1$, hence $\mathcal{O}_{\Gamma}(2) \cong \omega_{\Gamma}$. Computing the cohomology of the Mayer-Vietoris sequence ($C$ is linearly normal by inductive assumption)
$$0 \rightarrow \mathcal{I}_X(1) \rightarrow \widetilde{\mathcal{I}}_C(1) \oplus \mathcal{I}_D(1) \rightarrow \mathcal{I}_{\Delta}(1) \rightarrow 0$$
one obtains that $X$ is linearly normal, thus $\Gamma$ is too. It then follows from Fact \ref{ineq} that the pair $[\Gamma,\mathcal{O}_{\Gamma}(1)]$ is parameterized by a point of a component $V_{r+1}$ of the locus $\mathcal{S}^{r+1}_{g(r+1)}$ having codimension $\binom{(r+1)+1}{2}$ in $\mathcal{S}_{g(r+1)}$, and the inductive step is proved.
\end{proof}
Theorem \ref{ind} combined with Proposition \ref{farkasinductive} immediately gives the main result of the paper
\begin{teo}
For all integers $r \geq 2$ and $g \geq g(r) \doteq {r+2 \choose 2}$, the locus $\mathcal{S}^r_{g}$ has a component of codimension ${r+1 \choose 2}$ in $\mathcal{S}_{g}$.
\end{teo}
\end{section}
$\left.\right.$\\ \\
\textbf{Acknowledgements.}\\
The author wants to thank Claudio Fontanari and Letizia Pernigotti for having drawn his attention to the problem, and Edoardo Ballico and Edoardo Sernesi for helpful discussions.

\vspace{0.3cm}

\noindent
Luca Benzo \newline
Dipartimento di Matematica \newline
Universit\`a di Trento \newline
Via Sommarive 14 \newline
38123 Trento, Italy. \newline
E-mail address: luca.benzo@unitn.it
\end{document}